\documentclass[12pt]{amsart}

\usepackage{amsmath}
\usepackage[all]{xypic}
\usepackage[centering, hmargin=1.1in, vmargin=1.5in]{geometry}
\usepackage{amssymb}

\newtheorem{lemma}{Lemma}[section]
\newtheorem{theorem}[lemma]{Theorem}
\newtheorem{corollary}[lemma]{Corollary}
\newtheorem{example}[lemma]{Example}
\newtheorem{remark}[lemma]{Remark}
\newtheorem{prop}[lemma]{Proposition}

\theoremstyle{definition}
\newtheorem{definition}[lemma]{Definition}

\newcommand{\mr}{\mathrm}
\DeclareMathOperator{\kum}{Kum}

\author{Afsaneh Mehran}
\title{Kummer surfaces associated to (1,2)-polarized abelian surfaces}
\address{Section de Math\'ematiques, Universit\'e de Gen\`eve, 2-4 rue du Li\`evre, Case Postale 64, 1211 Gen\'eve 4, Suisse.\\}
\email{Afsaneh.Mehran@math.unige.ch}
\begin{document}
\maketitle
\begin{abstract}
The aim of this paper is to describe the geometry of the generic
Kummer surface associated to a $(1,2)$-polarized abelian surface.
We show that it is the double cover of a weak del Pezzo surface
and that it inherits from the del Pezzo surface an elliptic
fibration with twelve singular fibers of type $I_2.$
\end{abstract}

\section{Introduction}
The extensive study of Kummer surfaces is explained by their rich
geometry and their multiple roles in the theory of K3 surfaces and
beyond \cite{H1,torelli, B3}.

Let $A$ be an abelian surface and consider the involution which
maps $a$ to $-a$ for any $a$ in $A$. This involution has sixteen
fixed points, namely the sixteen two-torsion points of $A$. The
quotient surface has sixteen ordinary double points and its
minimal resolution is a smooth K3 surface called the Kummer
surface associated to $A$ and denoted by $\mr{Kum}(A)$. Nikulin
proved that any K3 surface containing sixteen disjoint smooth
rational curves is a Kummer surface \cite{N2}.

Given a Kummer surface $\mr{Kum}(A)$, there is a natural way of
constructing new Kummer surfaces from it. One takes the minimal
model of the double cover of $\mr{Kum}(A)$ branched along eight
disjoint smooth rational curves $C_1, \dots, C_8,$ that are even
(see section 2) and that are orthogonal in $\mr{Pic(Kum}(A))$ to
eight other smooth rational curves. We obtain in this way a new
Kummer surface $\mr{Kum}(B)$ together with a rational map
$\mr{Kum}(B) \stackrel{\tau}\dashrightarrow \mr{Kum}(A)$.

In the second section of the paper, we explain this construction
in details and show that the abelian surface associated to the new
Kummer surface $\mr{Kum}(B)$ is isogenous to $A$. In fact we prove
that the map $\tau$ is induced by an isogeny of degree two on the
associated abelian surfaces.

In section 3, we describe the geometry of a generic jacobian
Kummer surface and explain its classical double plane model. We
also recall a theorem of Naruki \cite{Naruki} giving explicit
generators of the N\'eron-Severi lattice of a generic jacobian
Kummer surface.

In section 4, we apply the construction of section 2 to the
generic jacobian Kummer surface. We obtain in this way, fifteen
non isomorphic Kummer surfaces which are associated to
$(1,2)$-polarized abelian surfaces.

Finally in section 5, we show  that the Kummer surfaces of section
4 admit an elliptic fibration with twelve singular fibers of the
type $I_2$. We also prove that these Kummer surfaces are double
cover of a week Del Pezzo surface (i.e. the blowup of $\mathbb
P^2$ at seven points) and that for each of our Kummer surfaces
there exists a decomposition of a very degenerate sextic $\mathcal
S$ (see figure \ref{plane sextic}) into a quartic $Q$ and a conic
$C$ for which we have the theorem

\begin{theorem}\label{San11}
The rational double cover $\mr{Kum}(B) \stackrel{\tau}\dashrightarrow \mr{Kum}(A)$ decomposes as
$$\xymatrix{\mr{Kum}(B) \ar@{-->}[d]^{\tau}\ar[r]^{\varphi}&
T \ar@{->}[d]^{\zeta}\\ \mr{Kum}(A) \ar@{->}[r]^{\phi}& \mathbb P^2}$$ where
$\phi$ is the canonical resolution of the double cover of $\mathbb
P^2$ branched along $\mathcal S$. The maps $\zeta$ and $\varphi$
are the canonical resolutions of the double covers branched along $Q$ and $\zeta^*(C)$ respectively.
\end{theorem}


\section{Even Eight and Kummer surface}

We will now introduce the notion of an even eight and the double
cover construction associated to it. By applying this construction
to special even eights of a Kummer surface, we obtain new Kummer
surfaces.
\begin{definition}\label{eveneight}
Let $Y$ be a K3 surface, an \textit{even eight} on $Y$ is a set of
eight disjoint smooth rational curves $C_1, \dots, C_8,$ for which
$C_1 + \cdots + C_8 \in 2S_Y.$ Here $S_Y$ denotes the
N\'eron-Severi group of $Y.$
\end{definition}

If $C_1, \dots, C_8,$ is an even eight on a K3 surface $Y$, then
there is a double cover $Z \stackrel{p}\to Y$, branched on $C_1 +
\dots +C_8$. If $E_i$ denotes the inverse image of $C_i$, then
$p^*(C_i)=2E_i$ and $E_i^2=-1$. Hence, we may blowdown the $E_i$'s
to the surface $X$ and obtain the diagram $$\xymatrix{Z
\ar@{->}[d]_{p }\ar[r]^{\epsilon} & X \ar@{-->}^{2:1}[ld]
\\Y & }$$
It turns out that the surface $X$ is again a K3 surface and the
covering involution $\iota: X \to X$ is symplectic with eight
fixed points \cite{N2}.

Suppose now that the K3 surface $Y$ is a Kummer surface, we want
to exhibit natural even eights lying on it. For this purpose, we
recall a central lemma of Nikulin.
\begin{lemma}\cite{N2}\label{nikulin}
Let $Y$ be a Kummer surface and let $E_1, \dots, E_{16} \subset Y$
be sixteen smooth disjoint rational curves. Denote by $I=\{ 1,
\dots, 16 \}$ the set of indices for the curves $E_{i}$'s and by
$Q = \{ M \subset I \vert \quad \frac{1}{2} \sum_{i\in M}E_{i}\in
S_Y\}$; then for every $M$ in $Q$, we have $\# |M|=8 \textrm{ or }
16$ and there exists on $I$ a unique 4-dimensional affine geometry
structure over $\mathbb F_{2}$, whose hyperplanes consist of the
subsets $M \in Q$ containing eight elements.
\end{lemma}

The existence of such a 4-dimensional affine geometry implies that
$I \in Q $ or equivalently that $\sum_{i=1}^{16}E_{i}\in 2 S_Y.$
We can proceed exactly as for an even eight and take the double
cover $V \stackrel{p}\to Y$ branched along $E_1+ \dots + E_{16}$.
Again we blowdown the preimage of the $E_i$'s to a surface $A$ and
obtain the diagram $$\xymatrix{V \ar@{->}[d]_{p }\ar[r]^{\epsilon}
& A \ar@{-->}^{\pi_A }[ld]
\\Y & }$$The difference with this diagram and the one above is
that now the surface $A$ is an abelian surface and that the map
$\pi_A$ realizes $Y$ as the Kummer surface associated to $A$. We
point out that by uniqueness, the affine geometry on $I$
corresponds to the one existing on $A_{2}$, the set of 2-torsion
points on $A$. \cite{N2}.

It follows also from the lemma that there exist on $Y$ ($\simeq
\mr{Kum}(A)$) thirty even eights, denoted by $M_1, \cdots,
M_{30}$, i.e. the thirty affine hyperplanes of $I$.

Let $M \in \{ M_1, \dots, M_{30} \}$ be one of these even eights.
We can assume that $M$ consists of the curves $E_{1}, \dots,
E_{8}$. The curves $E_{9}, \dots, E_{16}$ are then orthogonal to
$M$, i.e. $$E_i \cdot E_j =0 \textrm{ if } 1 \leq i \leq 8
\textrm{ and } 9\leq j \leq 16.$$ If $X
\stackrel{\tau}\dashrightarrow \mr{Kum}(A) $ is the double cover
associated to $M$, then the K3 surface $X$ contains again sixteen
disjoint smooth rational curves. Indeed since the curves $E_{9},
\dots, E_{16}$ do not intersect the branch locus of the double
cover $p: Z \to Y$, they split under $p$ and define sixteen
disjoint smooth rational curves on $Z$. These sixteen curves are
then isomorphically mapped by the blowdown $Z \stackrel{\epsilon}
\to X$ to sixteen curves on $X$. It follows that $X$ contains
sixteen disjoint smooth rational curves and hence it is a Kummer
surface.

\begin{prop}\label{prop}
Let $M$ be an even eight on a Kummer surface $\mr{Kum}(A)$ such as
above, then the K3 surface $X$ associated to $M$ is a Kummer
surface. Moreover there is an abelian surface $B$ associated to
$X$ for which we have the commutative diagram
$$\xymatrix{B\ar@{-->}[d]_{\pi_B } \ar[r]^{p} & A
\ar@{-->}[d]_{\pi_A}
\\X=\mr{Kum}(B) \ar@{-->}[r]^{\tau}& \mr{Kum}(A)}$$where $B\stackrel{p}\to
A$ is an isogeny of degree two.
\end{prop}

\begin{proof}
\noindent Since we have already shown that $X$ is a Kummer
surface, we only have to prove that $B$ is degree two isogenous to
$A$. Write the abelian surface $A$ as the complex torus $\mathbb
C^2 / \Lambda$ and let $E_{9}, \dots, E_{16}\subset \mr{Kum}(A)$
be the eight disjoint smooth rational curves orthogonal to $M$.
These curves also form an even eight and hence they correspond to
an affine hyperplane $H$ in $A_2$. Up to translation we can fix
the origin on $A$ in $H$. Let $\frac{[v]}{2}$ be the generator of
$A_2/H$, it defines a sublattice $\Lambda' \subset \Lambda$.
Explicitly we have that $\Lambda'= \mathbb Z h_1 \oplus \mathbb Z
h_2 \oplus \mathbb Z h_3 \oplus \mathbb Z 2v$, where $H=\langle
\frac{[h_1]}{2}, \frac{[h_2]}{2}, \frac{[h_3]}{2} \rangle \subset
A_2$.
The canonical inclusion $\Lambda' \hookrightarrow \Lambda$,
induces the following commutative diagram : $$\xymatrix{\mathbb
C^2/\Lambda'\ar@{-->}[d]_{\pi'} \ar[r]^{p} & \mathbb C^2/\Lambda
\ar@{-->}[d]_{\pi}\\
              \mr {\kum(\mathbb C^2/\Lambda')} \ar@{-->}[r]^{q}& \mr {\kum(\mathbb C^2/\Lambda)}}$$
where $p$ is an isogeny of degree two. The covering involution of
$p$ is given by the translation by the 2-torsion point $[v]$ in
$\mathbb C^2/\Lambda'$.
It induces the symplectic involution on $\mr {\kum(\mathbb
C^2/\Lambda')}$ $$\sigma: \mr {\kum(\mathbb C^2/\Lambda')} \to \mr
{\kum(\mathbb C^2/\Lambda')}$$which has exactly eight fixed points
\cite{N5}, namely
the projection of the sixteen points on $\mathbb C^2/\Lambda'$ satisfying$$[z]+[v]=-[z] \textrm{, or equivalently } 2[z]=[v].$$  
The isogeny $p$ maps the set $\{ [z] \in \mathbb C^2/\Lambda'
\textrm { }| \textrm { }2[z]=[v]\}$ to $\mr {A_2-H}$. In other
words, the affine hyperplane $A_2-H$ corresponds  to the even
eight $M$ in $\mr {Kum}(\mathbb C^2/\Lambda)$.  Hence the
resolution of the rational map $q$ is exactly the double cover of
$\mr{Kum}(A)$ branched along $M$ and the abelian surface $\mathbb
C /\Lambda'$ is $B$.

\end{proof}

\section{Jacobian Kummer surface}
In this section we briefly expose the classical geometry of a
jacobian Kummer surface and its beautiful $16_6$-configuration. We
describe its double plane model and give explicit generators for
its N\'eron-Severi lattice. This description follows a paper of N
aruki \cite{Naruki}.

A Kummer surface $\kum(A)$ is said to be a jacobian Kummer surface
if the surface $A$ is the jacobian of a curve $C$ of genus two.
Moreover, it is a generic jacobian Kummer surface if its Picard
rank is 17.

Recall that the degree two map given by the linear system $|2C|$,
$A \stackrel{|2C|} \to \mathbb P^3$, factors through the
involution $a \stackrel{i} \mapsto -a$, and hence defines an
embedding $A/\{ 1, i\} \hookrightarrow \mathbb P^3$. The image of
this map is a quartic $Y_0 \subset \mathbb P^3$ with sixteen
nodes. Denote by $L_0$ the class of a hyperplane section of $Y_0$.
Projecting $Y_0$ from a node defines a rational map $Y_0
\stackrel{2:1}\dashrightarrow \mathbb P^2$. We blowup the center
of projection

$$\xymatrix{Y_1\ar@{->}[d] \ar[rd] & \\
              Y_0 \ar@{-->}[r]& \mathbb P^2}$$and we call $E_1\subset Y_1$  the
              exceptional divisor and $L_1\subset
              Y_1$ the pullback of a line in $\mathbb P^2$.
              Finally we resolve the remaining fifteen
              singularities of $Y_1$ and obtain the Kummer surface
              $\mr{Kum}(A)$ and a map of degree two
$\mr{Kum}(A) \stackrel{\phi} \to \mathbb P^2$.  The map $\phi$ is
given by the linear system $|L-E_0|$, where $L$ and $E_0$ are the
pullback of $L_1$ and $E_1$ respectively.

The branch locus of the map $\phi$ is a reducible
 plane sextic $\mathcal S$, which is the union of six lines, $l_1, \cdots, l_6$, all tangent to a conic $W$.

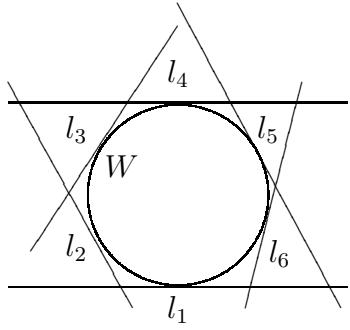
\begin{figure}[h]
$$
  \begin{xy}
   <0cm,0cm>;<1.5cm,0cm>:
    (2,-.3)*++!D\hbox{$l_1$},
    (1.1,0.25)*++!D\hbox{$l_2$},
    (1.1,1.3)*++!D\hbox{$l_3$},
    (2,1.8)*++!D\hbox{$l_4$},
    (2.8,1.3)*++!D\hbox{$l_5$},
    (2.9,0.2)*++!D\hbox{$l_6$},
    (1.5,1)*++!D\hbox{$W$},
    (.5,0.18);(3.5,0.18)**@{-},
    (0.5,2);(1.6,0)**@{-},
    (3.1,2);(2.6,0)**@{-},
    (0.7,.5);(2,2.5)**@{-},
    (0.5,1.82);(3.5,1.82)**@{-},
    (2,2.7);(3.45,0)**@{-},
    (2,1)*\xycircle(.8,.8){},
  \end{xy}
  $$
\caption{The sextic $\mathcal S$\label{plane sextic}}
\end{figure}

\noindent Let $p_{ij}=l_i \cap l_j \in \mathbb P^2$, where $1\leq
i < j \leq 6$. Index the ten $(3,3)$-partitions of the set $\{ 1,
2, \dots ,6 \}$, by the pair $(i,j)$ with $2 \leq i < j \leq 6$.
Each pair $(i,j)$ defines a plane conic $l_{ij}$ passing through
the sixtuplet $p_{1i}, p_{1j}, p_{ij},p_{lm},p_{ln},p_{mn}$, where
$\{ l,m,n \}$ is the complement of $\{ 1,i,j\}$ in $\{ 1, 2, \dots
,6 \}$ and where $l < m < n$. The map $\phi$ factors as
$$\mr{Kum}(A) \stackrel{\tilde{\phi}} \longrightarrow
\tilde{\mathbb P}^2\stackrel{\eta}\longrightarrow \mathbb
P^2$$where $\eta$ is the blowup of $\mathbb P^2$ at the $p_{ij}$'s
and where $\tilde{\phi}$ is the double cover of $\tilde{\mathbb
P}^2$ branched along the strict transform of the plane sextic
$\mathcal S$ in $\tilde{\mathbb P}^2$. Denote by $E_{ij} \subset
\mr{Kum}(A)$ the preimage of the exceptional curves of
$\tilde{\mathbb P}^2$. The ramification of the map $\tilde{\phi}$
consists of the union of six disjoint smooth rational curves $C_0+
C_{12} + C_{13}+ C_{14}+ C_{15}+ C_{16}$. The preimage of the ten
plane conics $l_{ij}$ defines ten more smooth disjoint rational
curves $C_{ij}\subset \mr{Kum}(A), 2 \leq i < j \leq 6$. Finally,
note that $\phi(E_0)=W$. The sixteen curves $E_0, E_{ij} \quad 2
\leq i < j \leq 6$ are called the \textit{nodes} of $ \mr{Kum}(A)$
and the sixteen curves $C_0, C_{ij}$, $2 \leq i < j \leq 6$ are
called the \textit{tropes} of $ \mr{Kum}(A)$. These two sets of
smooth rational curves satisfy a beautiful configuration called
the $16_6$-configuration, i.e. each node intersects exactly six
tropes and vice versa.

\noindent It is now possible to fully describe the N\'eron-Severi
lattice $S_{ \mr{Kum}(A)}$ of a general jacobian Kummer surface.

\begin{theorem}\cite{Naruki} Let $\mr{Kum}(A)$ be a generic jacobian Kummer surface. Its N\'eron-Severi lattice $S_{ \mr{Kum}(A)}$
is generated by the classes of $E_{0}, E_{ij}$, $C_{0}, C_{ij}$
and $L$, with the relations:

\begin{enumerate}
\item $C_{0}= \frac{1}{2}(L- E_{0} -
\sum_{i=2}^{6} E_{1i}),$
\item $C_{1j}= \frac{1}{2}(L - E_{0} - E_{1j}- \cdots - E_{j-1j}- E_{jj+1}- \cdots E_{j6}),$ where $2\leq j \leq 6,$
\item $C_{jk}= \frac{1}{2}(L - E_{1j} -E_{1k}- E_{jk}- E_{lm}-E_{ln}-
E_{mn})$ where $2\leq i < j \leq 6,$ and $\{ l, m, n\}$ are as
described above.

\end{enumerate}

\noindent The intersection pairing is given by:

\begin{enumerate}

\item  the $E_{0}, E_{ij}$ are mutually orthogonal,

\item $\langle L, L \rangle =4, \langle L, E_{0} \rangle =\langle L, E_{ij}\rangle =0,$

\item $\langle E_{0}, E_{0} \rangle= \langle E_{ij}, E_{ij} \rangle=-2$,

\item the $C_{0}, C_{ij}$ are mutually orthogonal,

\item  $\langle L, C_{0} \rangle=  \langle L, C_{ij} \rangle=2$.

\end{enumerate}


\noindent The action on $S_{ \mr{Kum}(A)}$ of the covering
involution $\alpha$ of the map $\phi$ is given by:

\begin{tabular}{lrrlr}

$\alpha (C_{0})=C_{0}$ & &&

$\alpha (C_{1j})=C_{1j}$ & $2 \leq j \leq 6$\\

$\alpha (E_{ij})=E_{ij}$ & $1 \leq i < j \leq 6$ &&

$\alpha (L)=3 \mathrm L - 4 E_{0}$ & \\

$\alpha (E_{0})=2 L - 3E_{0}$& &&

$\alpha (C_{ij})= C_{ij} + L - 2E_{0}$ & $2 \leq i < j \leq 6$.\\

\end{tabular}

\end{theorem}
\begin{remark}
The minimal resolution of the double cover of $\mathbb P^2$
branched along the sextic $\mathcal S$ in figure \ref{plane
sextic} is a Kummer surface (see \cite{H1} for a proof).\\
\end{remark}

\section{(1,2)-polarized Kummer surfaces}

In this section, we apply the construction of section 2 to a
generic jacobian Kummer surface. We identify all the even eights
made out of its nodes and study the associated Kummer surfaces.
First we recall some standard facts about the polarization of
abelian varieties.

A polarization on a complex torus $\mathbb C^g/ \Lambda$ is the
class of an ample line bundle $L$ in its the N\'eron-Severi group.
As the latter group is equal, for abelian varieties, to the group
of hermitian forms $H$ on $\mathbb C^g$, satisfying
$E=\mr{Im}H(\Lambda, \Lambda) \subset \mathbb Z$, the ample line
bundle $L$ corresponds to a positive definite hermitian forms
$E_L$. According to the elementary divisor theorem, there exists a
basis $\lambda_1, \dots, \lambda_g, \mu_1, \dots, \mu_g$ of
$\Lambda$, with respect to which $E_L$ is given by the matrix
$$\left(
\begin{array}{cc}0 & D \\ -D & 0\end{array}\right) \textrm{ with   } D=\left(
\begin{array}{cccc} d_1 & 0 & 0 & \ldots \\ 0& d_2 & 0&  \ldots \\ \vdots & 0&  \ddots& 0 \\ \vdots & \vdots & 0 &d_g \end{array}\right)$$
where $d_i \ge 0$ and $d_i | d_{i+1}$ for $i=1, \dots, g-1.$ The
vector $(d_1, d_2, \dots, d_g)$ is the type of the line bundle
$L$.

\begin{example}\cite{BL}\label{example}

\begin{enumerate}
\item If $J(C)$ is the Jacobian of a curve $C$ of genus two,
then
the line bundle associated to the divisor $C$ is a polarization of
type $(1,1)$.

\item If $L$ is a polarization of type $(d_1, \dots, d_g)$ on a complex torus, then
$\chi(L)=d_1 \cdot \cdot  \cdot d_g$.

\item If $X_1 \stackrel{p} \to X_2$ is an isogney of degree 2 of
abelian surfaces and $L$ is a polarization of type $(1,1)$ on
$X_2$, then $\chi(p^*(L))=2 \chi(L)=2\cdot 1$. Hence $p^*(L)$ is a
polarization of type $(1,2)$ on $X_1$.
\end{enumerate}
\end{example}


\begin{prop}\label{fifteen}
Let $\mr{Kum}(A)$ be a generic jacobian Kummer surface and let
$E_0, E_{ij},$ $1\leq i<j\leq 6$ be its sixteen nodes. There exist
fifteen even eights made out of its nodes that do not contain
$E_0$. These even eights are of the form
$$\Delta_{i,j}=E_{1i}+\cdots+ \widehat{E}_{ij}+ \cdots + E_{i6} +
E_{1j}+ \cdots + \widehat{E}_{ij} + \cdots + E_{j6},$$where$\quad
1 \leq i < j \leq 6$ and $E_{11}=0.$

The Kummer surface $\mr{Kum}(B_{ij})$ obtained from the double
cover branched along $\Delta_{ij}$ is associated to an abelian
surface $B_{ij}$ with a (1,2)-polarization.
\end{prop}
\begin{proof}
For any couple $(i,j)$ with $1 \leq i < j \leq 6$, consider the
divisor $2C_{1i} + 2C_{1j}$, where we set $C_{11}:= C_0$.
According to the description of the N\'eron-Severi lattice of a
general jacobian Kummer surface in the previous section, we have
the equality $$2C_{1i} + 2C_{1j}= 2(L-E_0)-(E_{1i}+ \cdots +
E_{ij}+ \cdots + E_{i6} + E_{1j}+ \cdots + E_{ij} + \cdots +
E_{j6}).$$ Therefore $$2C_{1i} + 2C_{1j} - 2(L-E_0)+2 E_{ij}=
E_{1i}+\cdots+ \widehat{E}_{ij}+ \cdots + E_{i6} + E_{1j}+ \cdots
+ \widehat{E}_{ij} + \cdots + E_{j6}$$ and consequently
$$E_{1i}+\cdots+ \widehat{E}_{ij}+ \cdots + E_{i6} + E_{1j}+
\cdots + \widehat{E}_{ij} + \cdots + E_{j6}$$is an even eight not
containing $E_0$. As there are exactly fifteen choices for $i$ and
$j$, we obtain this way all of the possible even eight.

Let $\mr{Kum}(B_{ij})$ be the Kummer surface obtained by taking
the double cover branched along such an even eight. By the
proposition \ref{prop}, the surface $B_{ij}$ is degree two
isogenous to $A$. Since $A$ has a $(1,1)$-polarization, it follows
from the example \ref{example} that $B_{ij}$ has a
$(1,2)$-polarization.
\end{proof}

The reason why we only consider the even eights not containing
$E_0$ is because we would obtain the exact same Kummer surface
whether we take the double cover branched along an even eight or
its complement (see proof of the proposition \ref{prop}).

In the remaining of this section, we will prove that no two of the
Kummer surfaces $\mr {Kum}(B_{ij})$ are isomorphic.

\begin{definition}\label{Nikulin}
The \textit{Nikulin lattice} is an even lattice $N$ of rank eight
generated by $\{ c_i\}_{i=1}^8$ and
$d=\frac{1}{2}\sum_{i=1}^8c_i$, with the bilinear form $c_i \cdot
c_j = - 2\delta_{ij}.$
\end{definition}
\begin{remark}
If $C_1, \dots, C_8$ is an even eight on a K3 surface, then the
primitive sublattice generated by the class of the $C_i$'s in the
N\'eron-Severi group of $X$ is a Nikulin lattice.
\end{remark}

The following proposition gives a condition on two even eights to
give rise to non isomorphic K3 surfaces.

\begin{prop}
Let $Y$ be a Kummer surface and let $\Delta_1$ and $\Delta_2$ be
two even eights on $Y$. Denote by $X_1$ and $X_2$ the respective
double covers of $Y$. If $N_1, N_2 \subset S_Y $ are the two
Nikulin lattices corresponding to $\Delta_1$ and $\Delta_2$, then
$$X_1 \simeq X_2 \iff \exists f \in \mr{Aut}(Y) \textrm{ such that
} f^*(N_1)=N_2.$$
\end{prop}

\begin{proof}

\noindent We suppose that $X_1$ is isomorphic to $X_2$ and we
denote by $X_2 \stackrel{g}\to X_1$ an isomorphism between $X_2$
and $X_1$. Let $X_1 \stackrel{i_1} \to X_1$ and $X_2
\stackrel{i_2} \to X_2$ be the covering involutions with respect
to the rational double covers $X_1 \stackrel{\tau_1}
\dashrightarrow Y$ and $X_2 \stackrel{\tau_2} \dashrightarrow Y$.

\noindent \textit{Claim:} The following diagram is commutative:
$$\xymatrix{H^2(X_1, \mathbb Z) \ar@{->}[r]^{g^*}\ar[d]^{i^*_1}&
H^2(X_2, \mathbb Z) \ar@{->}[d]^{i^*_2}\\ H^2(X_1, \mathbb Z)
\ar@{->}[r]^{g^*}& H^2(X_2, \mathbb Z).}$$

\noindent \textit{Proof of the claim:} Suppose that the above
diagram does not commute. Then the surface $X_1$ would admit two
distinct symplectic involutions, namely $i_1$ and $g \circ
i_2\circ g^{-1}$. Moreover the quotient of $X_1$ by both of these
involutions would be birational to the same Kummer surface $Y$. In
\cite{AM}, it is shown that the rational double cover of a Kummer
surface $\mr{Kum}(A)$ is determined by an embedding $T_X
\hookrightarrow T_A$ preserving the Hodge decomposition of $T_X$
and $T_A$. Since there is an unique embedding of $T_X$ into $T_A$
which preserves the Hodge decomposition, it follows that
$i_1=g^{-1} \circ i_2\circ g.$

Hence $i_2 \circ g = g \circ i_1$ and the isomorphism $g$ descends
to an isomorphism on the quotients  $$X_2/ i_2 \stackrel{g}\to
X_1/i_1.$$ Since this isomorphism maps the eight singular points
of $X_2/ i_2$ to the eight singular points of $X_1/ i_1$, it
extends to an automorphism $Y \stackrel{f}\to Y$, for which
$f^*(N_1)=N_2.$

Conversely, let $Y \stackrel{f} \to  Y$ be an automorphism of $Y$
for which $ f^*(N_1)=N_2$. Denote by $Z_i \stackrel{p_i}\to Y$ the
double cover of $Y$ branched along the even eight $N_i$ for
$i=1,2$. Consider the fiber product $$\xymatrix{Z_1 \times_Y Y
\ar@{->}[r]^{q}\ar[d]^{p}& Z_1 \ar@{->}[d]^{p_1}\\Y
\ar@{->}[r]^{f}&Y.}$$ The map $Z_1 \times_Y Y \stackrel{p} \to Y$
is a double cover of $Y$ branched along the even set $N_2$ or
equivalently $Z_1 \times_Y Y=Z_2$. Similarly, by considering the
fiber product $$\xymatrix{Z_2 \times_Y Y
\ar@{->}[r]^{h}\ar[d]^{r}& Z_2 \ar@{->}[d]^{p_2}\\Y
\ar@{->}[r]^{f}&Y,}$$ we see that $Z_2 \times_Y Y =Z_1$. The maps
$h$ and $q=h^{-1}$ define an isomorphism between $Z_1$ and $Z_2$
which induces  the required isomorphism between $X_1$ and $X_2$.
\end{proof}

Using the same notation as in the proposition \ref{fifteen}, we
prove the following theorem

\begin{prop}
Let $\Delta_{ij}$ and $\Delta_{i'j'}$ be two even eights defined
as in propostion \ref{fifteen}. $$\mr{Kum}(B_{ij}) \simeq
\mr{Kum}(B_{i'j'}) \Leftrightarrow \{i,j\}=\{i',j'\}.$$
\end{prop}
\begin{proof}
It is clear that if $\{i,j\}=\{i',j'\}$, then the corresponding
Kummer surfaces are equal. Thus we only have to prove the other
direction. Without loss of generality, we may assume that
$\Delta_{i'j'}=\Delta_{12}$ and we suppose that there exists $f$
an automorphism of $\mr{Kum}(A)$ for which
$f^*(\Delta_{12})=\Delta_{ij}$.

\noindent \textit{Claim}: $$\{
f^*(E_{13}),f^*(E_{14}),f^*(E_{15}),f^*(E_{16}),f^*(E_{23}),f^*(E_{24}),f^*(E_{25}),f^*(E_{26})\}=$$
$$ \{ E_{1i},\cdots, \hat{E_{ij}}, \cdots ,E_{i6},E_{1j},\cdots,
\hat{E_{ij}},\cdots, E_{j6}\}$$

\noindent\textit{Proof of the claim}: Let $N$ be a Nikulin lattice
and let $D \in N$ be a divisor represented by a smooth rational
curve.  Note that since $D$ is an effective reduced divisor and
$N$ is negative definite, then $D^2=-2$. It is therefore
sufficient to show that the only $-2$-classes in $N$ are the
$c_i$'s and the claim will follow. We write $D$ as $D=
\sum_{i=1}^8 \lambda_i c_i+\epsilon d$ where $\lambda_j \in
\mathbb Z$ and $\epsilon=0$ or $1$. If $\epsilon=1$, then the
equality $$D^2=-2\sum_{i=1}^8 \lambda_i^2-2\sum_{i=1}^8
\lambda_i-4=-2$$implies that $\sum_{i=1}^8
\lambda^2_i+\lambda_i=-1.$ Since the latter equation has no
integer solution, we conclude that $\epsilon=0$. Hence
$$D^2=-2\sum_{i=1}^8 \lambda^2_i=-2$$or equivalently,
$\sum_{i=1}^8 \lambda_i^2=1.$ Therefore there exists an unique
$\lambda_k$ for which $\lambda_k=1$ and $\lambda_i=0$ for $i\ne
k$.

In \cite{Keum2}, it is proven that any automorphism of a jacobian
generic Kummer surface induces $\pm \textrm{identity}$ on
$D_{S_{\mr{Kum}(A)}}$ where $D_{S_{\mr{Kum}(A)}}$ is the
discriminant group $S_{\mr{Kum}(A)}^*/S_{\mr{Kum}(A)}$. We want to
apply this fact to the automorphism $f$. We consider the action of
$f^*$ on the following two independent elements of
$D_{S_{\mr{Kum}(A)}}$ $$\frac{1}{2}(E_{13}+E_{14}+E_{23}+E_{24})
\textrm{ and } \frac{1}{2}(E_{12}+E_{23}+E_{15}+E_{35}).$$ From
the claim, we deduce that $$f^*(E_{13}+E_{14}+E_{23}+E_{24})=
E_{i_1i}+E_{i_2i}+E_{j_1j}+E_{j_2j}$$for some classes
$E_{i_1i},E_{i_2i},E_{j_1j},E_{j_2j} \in \Delta_{ij}$.

\noindent From the identity $f^*_{D_{S_{\mr{Kum}(A)}}}=\pm
\textrm{id}_{D_{S_{\mr{Kum}(A)}}},$ we also deduce that
$$f^*(\frac{1}{2}(E_{13}+E_{14}+E_{23}+E_{24}))= \pm
\frac{1}{2}(E_{13}+E_{14}+E_{23}+E_{24}).$$ Putting these two
informations together we find that $$E_{13}+E_{14}+E_{23}+E_{24}+
E_{i_1i}+E_{i_2i}+E_{j_1j}+E_{j_2j} \in 2S_Y.$$Since the only even
eights containing $E_{13}, E_{14},E_{23},E_{24}$ are $\Delta_{12}$
and $\Delta_{34}$, we deduce that $\Delta_{ij}=\Delta_{34}$. We
proceed similarly for $f^*(E_{12}+E_{23}+E_{13}+E_{35})$ and find
that $\Delta_{ij}$ must be equal to $\Delta_{25}$ which yields to
a contradiction.
\end{proof}

\begin{corollary}
The fifteen Kummer surfaces $\mr{Kum}(B_{ij})$ are not isomorphic.
\end{corollary}

\section{Elliptic Fibration and weak del Pezzo surface}

In this section, we provide an alternate description of the Kummer
surfaces $\mr{Kum}(B_{ij})$ as the double covers of a weak del
Pezzo surface. We relate this construction to the projective
double plane model of the generic jacobian Kummer surface of
section 3. First we note the existence on $\mr{Kum}(B_{ij})$ of an
elliptic fibration that will be useful later. For simplicity, we
will always argue for the Kummer surface $\mr{Kum}(B_{12})$.

\begin{prop}\label{fibration}
Let $\mr{Kum}(B_{12})$ be the Kummer surface constructed in the
proposition \ref{fifteen}. The surface $\mr{Kum}(B_{12})$ admits a
Weierstrass elliptic fibration with exactly twelve singular fibers
of the type $I_2$.
\end{prop}
\begin{proof}
Let $\mr{Kum}(A) \stackrel{\phi} \to \mathbb P^2$ be the double
plane model of the generic jacobian Kummer surface introduced in
section 3. Consider the pencil of lines passing through the point
$p_{12}$ in $\mathbb P^2$. Its preimage in $\mr{Kum}(A)$ defines
an elliptic fibration, given by the divisor class
$F=L-E_0-E_{12}.$ The divisors $$F_1=E_{15}+ E_{16}+2C_{0}+
E_{13}+ E_{14}, \quad \textrm{and} \quad F_2=E_{25}+
E_{26}+2C_{12}+ E_{23}+ E_{24}$$ define two fibers of type $I^*_0$
of this fibration. Moreover, the six divisors
\begin{center}
$F_3= L-E_0-E_{12}-E_{45}+E_{45}$,

$F_4= L-E_0-E_{12}-E_{46}+E_{46}$,

$F_5= L-E_0-E_{12}-E_{35}+E_{35}$,

$F_6= L-E_0-E_{12}-E_{36}+E_{36}$,

$F_7= L-E_0-E_{12}-E_{34}+E_{34}$,

$F_8= L-E_0-E_{12}-E_{56}+E_{56}$
\end{center}
\noindent define six $I_2$ fibers. Since the Euler characteristics
of the $F_i$'s add up to 24, which is equal to the Euler
characteristic of a K3 surface, we conclude by Shioda's formula
\cite{Shioda} that the $F_i$'s are the only singular fibers of the
elliptic fibration defined by the linear system $|F|$. Note also
that the curves $C_{13}$, $C_{14}$, $C_{15}$ and $C_{16}$ are
sections of this fibration.

We now analyze the induced fibration $\tau^*F$ on
$\mr{Kum}(B_{12})$, where
$\mr{Kum}(B_{12})\stackrel{\tau}\dashrightarrow \mr{Kum}(A)$ is
the rational double cover defined by the even eight $\Delta_{12}.$
We remark that the even eight $\Delta_{12}$ satisfies
$$\Delta_{12} = F_1 +F_2 -2(C_0 +C_{12})$$ which mean that the
eight components of $\Delta_{12}$ are exactly the eight components
of the fibers $F_1$ and $F_{2}$ that appear with multiplicity one.
Hence $\tau^*F_1$ and $\tau^*F_2$ are just smooth elliptic curves.
However the six fibers $F_3, \dots, F_8$ split under the cover and
define twelve $I_2$ fibers of the elliptic fibration on
$\mr{Kum}(B_{12})$ defined by $\tau^*F$. Again a computation of
Euler characteristics shows that these twelve $I_2$ fibers are the
only singular fibers of the linear system $|\tau^*F|$. Also the
sections $C_{13}$, $C_{14}$, $C_{15}$ and $C_{16}$ of $|F|$ pull
back to sections of $\tau^*F$, which is therefore a Weierstrass
elliptic fibration.

\end{proof}

We now proceed to the realization of the surface
$\mr{Kum}(B_{12})$ as a double cover of a weak del Pezzo surface.
We decompose the sextic $\mathcal S$ (see figure \ref{plane
sextic}) into the quartic $Q=l_3+l_4+l_5+l_6$ and the conic
$C=l_1+l_2$.
\begin{theorem}\label{San11}
The rational double cover associated to $\Delta_{12}$,
$\mr{Kum}(B_{12}) \stackrel{\tau}\dashrightarrow Y$ decomposes as
$$\xymatrix{\mr{Kum}(B_{12}) \ar@{-->}[d]^{\tau}\ar[r]^{\varphi}&
T \ar@{->}[d]^{\zeta}\\ Y \ar@{->}[r]^{\phi}& \mathbb P^2}$$ where
$\phi$ is the canonical resolution of the double cover of $\mathbb
P^2$ branched along $\mathcal S$. The maps $\zeta$ and $\varphi$
are the canonical resolutions of the double covers branched along
$Q$ and $\zeta^*(C)$ respectively.
\end{theorem}
\begin{proof}

Let $T_0 \to \mathbb P^2$ be the double cover of $\mathbb P^2$
ramified over the reducible quartic $Q$. Its canonical resolution
induces the diagram $$\xymatrix{T \ar@{->}[d]
\ar[r]\ar[rd]^{\zeta}& T_0 \ar@{->}[d]\\ \tilde{\mathbb P}^2
\ar@{->}[r]& \mathbb P^2}$$ where $\tilde{\mathbb P}^2 \to \mathbb
P^2$ is the the blowup of $\mathbb P^2$ at the six singular points
of $Q$. The surface $T$ is a non-minimal rational surface
containing six disjoint smooth rational curves. Indeed by Hurwitz
formula, the canonical divisor of $T$ is given by

$$K_{T}= \zeta^*(K_{\mathbb
P^2}+\frac{1}{2}(l_3+l_4+l_5+l_6))=-\zeta^*(H)$$where $H$ is a
hyperplane section. Thus $K_T^2=2, H^2 =2$ and $P_2(T)=0$. Denote
by $\tilde{Q}$ the proper transform of $Q$ in $T$. Using the
additivity of the topological Euler characteristic and the Noether
formula, we have that $$e(T)= e(T- \tilde{Q}) + e(\tilde{Q})=10
\Rightarrow \mathcal X(\mathcal O_T)=1 \Rightarrow q(T)=0$$ By
Castelnuovo's rationality criterion, $T$ is a rational surface. In
fact, we show that $T$ is a weak del Pezzo surface of degree two,
i.e. the blow up of $\mathbb P^2$ at seven points with nef
canonical divisor. Indeed we successively blown down the preimages
in $T$ of the four lines $l_3, l_4, l_5$ and $l_6$ as well as the
preimages in $T$ of the three ``diagonals'' of the complete
quadrangle formed by $l_3, l_4, l_5, l_6$. The surface obtained
after these seven blow down is a projective plane.

Consider the following curves of $T$
$$\zeta^*(C)=\zeta^*(l_1+l_2)= E_{1} + E_{2}, \textrm{ where } E_1
\textrm{ and } E_2 \textrm{ are smooth elliptic curves}$$ and
$$\zeta^*(W)=W_1+W_2 \textrm{ where } W_1 \textrm{ and } W_2
\textrm{ are smooth rational curves}$$ with the following
intersection properties:
$$E_i^2=2,  \quad W_i^2=0, \quad E_1 \cdot E_2=2, \quad W_1 \cdot
W_2=4, \quad W_i \cdot E_j=2 \quad \textrm{ for}\quad i\neq j.$$
(recall that $W$ is the plane conic tangent to the six lines $l_1,
\cdots,l_6$). The linear system $|E_1|$ defines an elliptic
fibration on $T$ with six singular fibers of type $I_2$. Take the
double cover branched along the two fibers $E_1 + E_2 \in
2\mr{Pic(T)}$. It induces the canonical resolution commutative
diagram: $$\xymatrix{X \ar@{->}[d]\ar[r] \ar[dr]^{\varphi}& X_0
\ar@{->}[d]\\ \tilde{T} \ar@{->}[r]& T}$$ where $\tilde{T} \to T$
is the blowup of $T$ at the two singular points of $E_1+E_2$.

\textit{Claim:} X is a Kummer surface.

\textit{Proof of the Claim:} Clearly $\mathcal
K_{X}=\varphi^*(\zeta^*(-H)+ \frac{1}{2}(E_1+E_2))=\mathcal O_X$.


1) The pullback by $\varphi$ of the six exceptional curves on $T$ define
twelve disjoint smooth rational curves on $X$.

2) The two exceptional curves of $X$ give two more rational curves
disjoint from 1).

3) Let $\varphi^*(W_1)=W'_1+W''_1$ and $\varphi^*(W_2)=W'_2+W''_2$
and let $\sigma$ be the lift on $X$ of the covering involution of
$\zeta$, then $\sigma(W'_1)=W'_2$ or $\sigma(W'_1)=W''_2$. Without
loss of generality, we can assume that $\sigma(W'_1)=W'_2$ and
hence get the following intersection numbers $$W'^2_i=W''^2_i=-2,
\quad W'_i\cdot W''_i=2 \quad \textrm{for } i=1,2 $$ $$ W'_1\cdot
W'_2=W''_1\cdot W''_2=4 \quad \textrm{and} \quad W'_1\cdot
W''_2=W''_1\cdot W'_2=0.$$ One easily checks that $W'_1$ and
$W''_2$ do not intersect the fourteen curves from 1) and 2).

\noindent In particular, the $K3$ surface $X$ contains sixteen
disjoint smooth rational curves. Consequently $X$ is a Kummer
surface.
 Moreover, the surface $X$ contains an elliptic fibration with twelve $I_2$ fibers. It
 also admits two non symplectic involutions $\theta$ and $\sigma$ where $\theta$ is the covering involution of
 the map $\varphi$ and $\sigma$ is the lift of the covering involution of $\zeta$ on $T$ encountered earlier.
 The composition $\iota = \varphi \circ \sigma$ defines a symplectic involution on $X$ whose quotient is a $K3$
 surface admitting an elliptic fibration with singular fibers identical to the one defined by $F$ on $Y$ in the proposition
 \ref{fibration}.

 In fact, we can now recover sixteen disjoint rational curves on the quotient and conclude that
 it is our original general Kummer surface $Y$ and that $X \simeq \mr{Kum}(B_{12})$.

\end{proof}


\end{document}